\documentclass[a4paper]{amsart}
\usepackage{amsmath}
\usepackage{amssymb}
\usepackage{bbm}
\usepackage{a4wide}

\parskip\smallskipamount

\usepackage{graphicx}

\usepackage{amssymb}
 \usepackage{amsthm}

\usepackage{amsmath,amsfonts,amsxtra}
\usepackage{tabularx}
\usepackage{multirow}

\newtheorem{theorem}{Theorem}

\newtheorem{proposition}[theorem]{Proposition}
\newtheorem{corollary}[theorem]{Corollary}

\newtheorem{definition}{Definition}

\newcommand{\N}{\mathbb{N}}
\newcommand{\TM}{{T_{\mathbf{M}}}}
\newcommand{\TP}{{T_{\mathbf{P}}}}
\newcommand{\TL}{{T_{\mathbf{L}}}}
\newcommand{\TD}{{T_{\mathbf{D}}}}

\newcommand{\TSS}[1]{{T^{\mathbf{SS}}_{#1}}}
\newcommand{\TDo}[1]{{T^{\mathbf{D}}_{#1}}}
\newcommand{\TAA}[1]{{T^{\mathbf{AA}}_{#1}}}

\newcommand{\TF}[1]{{T^{\mathbf{F}}_{#1}}}

\newcommand{\THam}[1]{{T^{\mathbf{H}}_{#1}}}

\newcommand{\TMT}[1]{{T^{\mathbf{MT}}_{#1}}}

\newcommand{\TSW}[1]{{T^{\mathbf{SW}}_{#1}}}
\newcommand{\TY}[1]{{T^{\mathbf{Y}}_{#1}}}
\newcommand{\SSW}[1]{{S^{\mathbf{SW}}_{#1}}}

\newcommand{\mr}{17 + 12 \sqrt2}

\newcommand{\Ran}{\operatorname*{Ran}}

\newcommand{\Lu}{\L uka\-sie\-wicz }

\newcommand{\R}{\mathbb{R}}

\def\lint#1{\left[{#1}\right[}
\def\rint#1{\left]{#1}\right]}
\def\cint#1{\left[{#1}\right]}
\def\opint#1{\left]{#1}\right[}
\def\uint{[0,1]}
\def\ouint{\left]0,1\right[}

\newenvironment{ilist}{\begin{list}{(\roman{enumi})}{\usecounter{enumi}%
                        \setlength{\listparindent}{0pt}%
                        \setlength{\labelsep}{0.4em}%
                        \setlength{\labelwidth}{2.1em}
                        \setlength{\leftmargin}{2.5em}%
                        \setlength{\rightmargin}{0pt}}}{\end{list}}

\begin{document}

% Title, authors and addresses

\title{Dominance in the family of Sugeno-Weber t-norms}

\author{Manuel KAUERS}
\address{%
Research Institute for Symbolic Computation\\
Johannes Kepler University Linz, Altenbergerstrasse 69, A-4040 Linz, Austria}
\email{mkauers@risc.uni-linz.ac.at}

\author{Veronika PILLWEIN}
\address{%
Research Institute for Symbolic Computation\\
Johannes Kepler University Linz, Altenbergerstrasse 69, A-4040 Linz, Austria}
\email{vpillwein@risc.uni-linz.ac.at}

\author{Susanne SAMINGER-PLATZ}
%\corauth[cor]{Corresponding author.}
\address{%
Department of Knowledge-Based Mathematical Systems,\\
Johannes Kepler University Linz, Altenbergerstrasse 69, A-4040 Linz, Austria}
\email{susanne.saminger-platz@jku.at}

\begin{abstract}
The dominance relationship between two members of the family of Sugeno Weber t-norms is proven by using a quantifer elimination algorithm.  Further it is shown that dominance is a transitive, and therefore also an order relation, on this family of t-norms.
\end{abstract}

\maketitle

% main text
\section{Introduction}
Dominance is a functional inequality which arises in different application fields. It most often appears when discussing the preservation of properties during (dis-)aggregation processes like, e.g., in flexible querying, preference modelling or computer-assisted assessment~\cite{DeBMes98,DiazMontesDeBaets07,Saminger05,SamingerMesiarBodenhofer02}. It is further crucial in the construction of Cartesian products of probabilistic metric and normed spaces~\cite{LafuerzaGuillen04,SchSkl83,Tar76} as well as when constructing many-valued equivalence and order relations~\cite{Bodenhofer99,DeBMes98,DeBaetsMesiar02,Val85}.

Introduced in 1976 in the framework of probabilistic metric spaces as an inequality involving two triangle functions (see~\cite{Tar76} and~\cite{SchSkl83} for an early generalization to operations on a partially ordered set), it was soon clear that dominance constitutes a reflexive and antisymmetric relation on the set of all t-norms. That it is not a transitive relation has been proven much later in 2007 in~\cite{Sarkoci08}. This negative answer to a long open question has, to some extent, been surprising. In particular since earlier results showed that for several important single-parametric families of t-norms, dominance is also a transitive and therefore an order relation~\cite{KleMesPap00,SamingerDeBaetsDeMeyer06,SamingerSarkociDeBaets06,Sarkoci05,She84} (see also~\cite{Saminger09} for an overview and additional results on families of t-norms and copulas).

The family of Sugeno-Weber t-norms has been one of the more prominent families of t-norms for which the dominance has not been completely characterized so far. First partial results could be achieved in~\cite{SamingerDeBaetsDeMeyer09} by invoking results on different sufficient conditions derived from a generalization of the Mulholland inequality and involving the additive generators of the t-norms, their pseudo-inverses and their derivatives (for more details on the differential conditions see~\cite{SamingerDeBaetsDeMeyer09}, for the generalization of the Mulholland inequality look at~\cite{SamingerDeBaetsDeMeyer08}).

The purpose of this paper is to close this gap. We present a proof for a complete characterization of dominance in the family of Sugeno-Weber t-norms.
This is interesting because from all the families of t-norms discussed in Section~4 in the monograph~\cite{KleMesPap00}, the family of Sugeno-Weber t-norms was, until now, the only family for which no complete classification result was available. But there are further aspects which make our results interesting:

First, the solution sets are of a completely different form than witnessed before for other families. So far dominance in single-parametric families has either been in complete accordance with the ordering in the family (i.e., dominance constitutes a linear order on the family of t-norms) or dominance has rarely appeared between family members, i.e., holds only in the trivial cases of self-dominance or when involving maximal or minimal elements of the family. For the family of Sugeno-Weber t-norms neither is the case.

Second, although the solution sets look different, dominance is a transitive relation on the family of Sugeno-Weber t-norms.

Third, the results have been achieved by the use of symbolic computation algorithms. More explicitly  a quantifier elimination algorithm for real closed fields (Cylindrical Algebraic Decomposition) has, after several transformation steps, been applied to logical equivalent formulations of the original problems. The present contribution is therefore also an example of a successful application of computer algebra and symbolic computation for solving polynomial inequalities.

The following preliminaries shall clarify the necessary notions and summarize basic facts about the family of Sugeno-Weber t-norms. Dominance as well as some basic aspects of quantifier elimination algorithms will be explained. We then provide and prove the main results --- the characterization of dominance between two t-norms of the family of Sugeno-Weber t-norms and transitivity of dominance in the family. We finally discuss the results in more detail.

\section{Preliminaries}
\subsection{Triangular norms}
We briefly summarize some basic properties of t-norms for a thorough understanding of this paper. Excellent overviews on and discussions of triangular norms (including historical accounts, further details, proofs, and references) can be found in the monographs~\cite{AlsinaFrankSchweizer06,KleMesPap00}, the edited volume~\cite{KleMes05} and the articles~\cite{KleMesPap04b,KleMesPap04c,KleMesPap04d}.

\begin{definition}\label{def:TNorm}
A \emph{triangular norm} (briefly \emph{t-norm}) $T\colon\uint^2\to \uint$ is a binary operation on the unit interval which is commutative, associative, increasing and has neutral element $1$.
\end{definition}

Speaking in more algebraic terms, $T$ turns the unit interval into an ordered Abelian semigroup whose neutral element is 1.

The most prominent examples of t-norms are the \emph{minimum} $\TM$, the \emph{product} $\TP$,  the \emph{\Lu t-norm} $\TL$ and the \emph{drastic product} $\TD$.
They are defined by $\TM(u,v)= \min(u,v)$, $\TP(u,v)=u\cdot v$, $\TL(u,v)= \max(u+v-1,0)$, and
\[
\TD(u,v)=
\begin{cases}
\min(u,v),&\text{if }\max(u,v)=1,\\
0, &\text{otherwise.}
\end{cases}
\]

Obviously, the basic t-norms $\TM, \TP$ and $\TL$ are continuous, whereas the drastic product $\TD$ is not. The comparison of two t-norms is done pointwisely, i.e., if, for all $x,y\in\uint$, it holds that $T_1(x,y)\ge T_2(x,y)$, then we say that $T_1$ is \emph{stronger} than $T_2$ and denote it by $T_1\ge T_2$. The minimum $\TM$ is the strongest of all t-norms, the drastic product $\TD$ is the weakest of all t-norms. Moreover, the four basic t-norms are ordered in the following way: $\TD \leq \TL \leq \TP \leq \TM$.

\begin{definition}
A t-norm $T$ is called
\begin{ilist}
\item \emph{Archimedean} if for all $u,v\in\ouint$ there exists an $n\in\N$ such that
\[
T(\underbrace{u,\ldots, u}_{n \text{ times}})<v\,.
\]
\item A t-norm $T$ is called \emph{strict} if it is continuous and strictly monotone, i.e., for all $u,v,w\in\uint$ it holds that
\[
T(u,v)<T(u,w)\quad \text{whenever} \quad u>0 \text{ and }v<w\,.
\]
\item A t-norm $T$ is called \emph{nilpotent} if it is continuous and if each $u\in\ouint$ is a nilpotent element of $T$, i.e., there exists some $n\in\N$ such that
\[
T(\underbrace{u,\ldots, u}_{n \text{ times}})=0\,.
\]
\end{ilist}
\end{definition}

Note that for a strict t-norm $T$ it holds that $T(u,v)>0$ for all $u,v\in \,]0,1]$, while for a nilpotent t-norm $T$ it holds that for every $u\in\ouint$ there exists some $v\in\ouint$ such that
$T(u,v)=0$ (each $u\in\ouint$ is a so-called \emph{zero divisor}). Therefore for a nilpotent t-norm $T_1$ and a strict t-norm $T_2$ it can never hold that $T_1\ge T_2$.

\subsubsection{The family of Sugeno-Weber t-norms}
In 1983 S.\ Weber proposed the use of this particular family for modelling the intersection of fuzzy sets~\cite{Web83}. Since then, its dual operations, the Sugeno-Weber t-conorms, defined for all $\lambda\in\cint{0,\infty}$ and all $u,v\in\uint$ by $\SSW{\lambda}(u,v)=1-\TSW{\lambda}(1-u,1-v)$, have played a prominent role for generalized decomposable measures ~\cite{KleWeb91,KleWeb99,Pap02,Web84}, in particular, since they already appeared as possible generalized additions in the context of $\lambda$-fuzzy measures in~\cite{Sug74}.

The family of Sugeno-Weber t-norms $(\TSW{\lambda})_{\lambda\in\cint{0,\infty}}$ is, for all $u,v\in\uint$, given by
\[
\TSW{\lambda}(u,v)=
    \begin{cases}
    \TP(u,v), &\text{if } \lambda=0,\\
    \TD(u,v), &\text{if }\lambda=\infty,\\
    \max(0, (1-\lambda)uv +\lambda (u+v-1)), &\text{if }\lambda\in\opint{0,\infty}.
    \end{cases}
\]
The family is of particular interest since all but two of its members are nilpotent t-norms.
Parameters $\lambda\in \opint{0,\infty}$ lead to nilpotent t-norms (with $\TSW{1}=\TL$ as special case), while $\TSW{0}=\TP$ is the only strict member. For $\lambda\in \lint{0,\infty}$, the Sugeno-Weber t-norms are continuous Archimedean t-norms~\cite{May94,Sug74,Web83}, for $\lambda\in\cint{0,1}$ they can be interpreted as convex combinations of $\TL$ and $\TP$, and are therefore also copulas (for more details on copulas see also~\cite{Nelsen06}).

For $\lambda\in\ouint$, an interesting characterization of the family members has been provided in~\cite{May94}: A t-norm $T$ is a Sugeno-Weber t-norm $\TSW{\lambda}$ with $\lambda\in\ouint$ if and only if $T$ is nilpotent, has an additive generator such that for each $w\in\uint$, the graph of the vertical section $T(x,.)$ of $T$ is a straight line segment (for more details on the additive generators and construction methods for t-norms we again refer to the  monographs~\cite{AlsinaFrankSchweizer06,KleMesPap00}, the edited volume~\cite{KleMes05} and the articles~\cite{KleMesPap04b,KleMesPap04c,KleMesPap04d}).

\subsection{Dominance}
The dominance relation has, as t-norms do, its roots in the field of probabilistic metric spaces~\cite{SchSkl83,Tar76}. It was originally introduced for associative operations (with common neutral element) on a partially ordered set~\cite{SchSkl83}, and has been further investigated for t-norms~\cite{SamingerDeBaetsDeMeyer06,Sarkoci05,Sarkoci08,Tar84} and aggregation functions~\cite{Saminger05,SamingerMesiarBodenhofer02,MesiarSaminger04}. For more recent results on dominance between triangle functions resp.~operations on distance distribution functions see also~\cite{SamingerSempi0x}.

We state the definition for t-norms only.

\begin{definition}
Consider two t-norms $T_1$ and $T_2$. We say that $T_1$ \emph{dominates} $T_2$ (or $T_2$ is dominated by $T_1$), denoted by $T_1\gg T_2$, if, for all $x,y,u,v \in\uint$, it holds that
\begin{equation}\label{eg:DefDom}
T_1(T_2(x,y),T_2(u,v))\geq T_2(T_1(x,u),T_1(y,v))\,.
\end{equation}
\end{definition}

As mentioned already earlier, the dominance relation, in particular between t-norms, plays an important role in various topics, such as the construction of Cartesian
products of probabilistic metric and normed spaces~\cite{LafuerzaGuillen04,SchSkl83,Tar76}, the construction of many-valued equivalence relations~\cite{DeBMes98,DeBaetsMesiar02,Val85} and many-valued order relations~\cite{Bodenhofer99}, the preservation of various properties during (dis-)aggregation processes in flexible querying, preference modelling and computer-assisted
assessment~\cite{DeBMes98,DiazMontesDeBaets07,Saminger05,SamingerMesiarBodenhofer02}.

Every t-norm, in fact every function non-decreasing in each of its arguments, is dominated by $\TM$. Moreover, every t-norm dominates itself and $\TD$. Since all t-norms have neutral element $1$, dominance between two t-norms implies their comparability: $T_1\gg T_2$ implies $T_1\ge T_2$. The converse does not hold.

Due to the induced comparability it also follows that dominance is an antisymmetric relation on the class of t-norms. Associativity and symmetry ensure that dominance is also reflexive on the class of t-norms.

Although dominance is not a transitive relation on the set of continuous, and therefore also not on the set of all, t-norms (see the results by Sarkoci~\cite{Sarkoci08,Sarkoci0x} and also~\cite{SamingerSarkociDeBaets06}), it is transitive on several single-parameteric families of t-norms (see also Table~\ref{tab:families}).

It is interesting to see that in all the cases displayed in Table~\ref{tab:families}, dominance is either in complete accordance with the ordering in the family (i.e., dominance constitutes a linear order on the family of t-norms) or dominance rarely appears among family members, i.e., holds only in the trivial cases of self-dominance and dominance involving maximal or minimal elements of the family (for an overview on known results and referential details see~\cite{Saminger09}).

\begin{table}
\begin{center}
%\small
\begin{tabular}{|l|c|c|}\hline
        Family of t-norms&$T_\lambda\gg T_\mu$&  Hasse-Diagrams\\[1ex]\hline
        &&\\
        Schweizer-Sklar
        $(\TSS{\lambda})_{\lambda\in\cint{-\infty,\infty}}$&$\lambda\le\mu$&
            \multirow{11}{*}{\includegraphics[height=3cm]{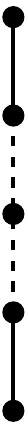}} \\
            \hspace*{10ex}{\tiny (Sherwood, 1984)}&&\\[1ex]
        Acz\'{e}l-Alsina
        $(\TAA{\lambda})_{\lambda\in\cint{0,\infty}}$&\multirow{4}{*}{$\lambda\ge\mu$}& \\
        Dombi $(\TDo{\lambda})_{\lambda\in\cint{0,\infty}}$&&\\
        Yager $(\TY{\lambda})_{\lambda\in\cint{0,\infty}}$&&\\
            \hspace*{10ex} {\tiny (Klement, Mesiar, Pap, 2000)}&&\\[1ex]
        $(T_\lambda^\mathbf{8})_{\lambda\in\cint{0,\infty}}$ &\multirow{4}{*}{$\lambda\ge\mu$}&\\
        $(T_\lambda^\mathbf{15})_{\lambda\in\cint{0,\infty}}$ &&\\
        $(T_\lambda^\mathbf{22})_{\lambda\in\cint{0,\infty}}$ &&\\
        $(T_\lambda^\mathbf{23})_{\lambda\in\cint{0,\infty}}$ &&\\
            \hspace*{10ex} {\tiny (Saminger-Platz, 2009)}&&\\[1ex]\hline
            &&\\
        Frank $(\TF{\lambda})_{\lambda\in\cint{0,\infty}}$&
            \multirow{3}{*}{$\lambda=0$, $\lambda=\mu$, $\mu=\infty$} &
            \multirow{5}{*}{\includegraphics[height=1cm]{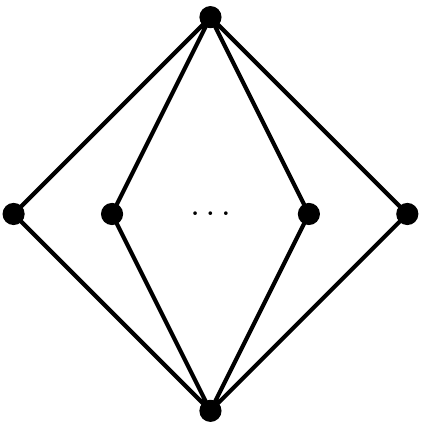}}\\
        Hamacher $(\THam{\lambda})_{\lambda\in\cint{0,\infty}}$&&\\
             \hspace*{10ex}{\tiny (Sarkoci, 2005)}&&\\[1ex]
        $(T_\lambda^\mathbf{9})_{\lambda\in\cint{0,\infty}}$ &
            \multirow{2}{*}{$\lambda=\infty$, $\lambda=\mu$, $\mu=0$} &\\
            \hspace*{10ex} {\tiny (Saminger-Platz, 2009)}&&\\[1ex]\hline
            &&\\
        \noindent Mayor-Torrens $(\TMT{\lambda})_{\lambda\in\uint}$&
            \multirow{3}{*}{$\lambda=0,\lambda=\mu$}&
            \multirow{3}{*}{\includegraphics[height=0.75cm]{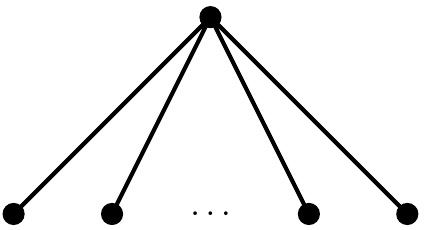}}\\
        Dubois-Prade $(T_\lambda^{\mathbf{DP}})_{\lambda\in\uint}$&&\\
            \hspace*{10ex}{\tiny (Saminger, De Baets, De Meyer, 2005)}&&\\[1ex]\hline
        \end{tabular}
\end{center}
\caption{Dominance relation in several families of t-norms}
\label{tab:families}
\end{table}

We will show below that dominance between two members of the family of Sugeno-Weber t-norms is of a complete different type, although finally dominance also turns out to be a transitive relation on this family.

\subsubsection{The family of Sugeno-Weber t-norms}
The members of the family form a decreasing sequence of t-norms with respect to their parameter, i.e., $\TSW{\lambda}\geq \TSW{\mu}$ if and only if $\lambda\leq \mu$.  Since dominance induces order on the t-norms involved, it is therefore clear that a necessary condition for $\TSW{\lambda}\gg\TSW{\mu}$ is that $\lambda\leq \mu$.

Dominance among the family members has further been studied in~\cite{SamingerDeBaetsDeMeyer09} by invoking results on different sufficient conditions derived from a generalization of the Mulholland inequality and involving the additive generators of the t-norms, their pseudo-inverses and their derivatives (for more details on the differential conditions see~\cite{SamingerDeBaetsDeMeyer09}, for the generalization of the Mulholland inequality see~\cite{SamingerDeBaetsDeMeyer08}). The results obtained did not lead to a full characterization of dominance in the whole family, but already indicated that the dominance structure might be of a completely different structure than the dominance relationship laid bare in any other family before. We quote the result from~\cite{SamingerDeBaetsDeMeyer09}.

%\pagebreak
\begin{proposition}\label{cor:suffcond}
Consider the family of Sugeno-Weber t-norms $(\TSW{\lambda})_{\lambda\in\cint{0,\infty}}$. For all $\lambda,\mu\in\cint{0,\infty}$
such that one of the following condition holds
\begin{itemize}
\item[(i)] $\lambda \le \min(1,\mu)$,
\item[(ii)] $1<\lambda\leq\mu\leq r^*$, with $r^*=6.00914$ denoting the second root of $\log^2(t)+\log(t)-t+1$,
\end{itemize}
it follows that $\TSW{\lambda}\gg \TSW{\mu}$.
\end{proposition}

\subsection{Quantifier Elimination and Cylindrical Algebraic Decomposition}
In contrast to many other families of t-norms, the dominance relation for Sugeno-Weber t-norms does not involve any logarithms or exponentials but can be formulated by addition, multiplication and the~$\max$-operation only. This is a striking structural advantage, because there are algorithms available for proving this kind of formulas automatically. The first algorithm for proving formulas about polynomial inequalities was already given by Tarski in the early 1950s~\cite{Tarski51} but his algorithm was only of theoretical interest. Nowadays, modern implementations~\cite{Brown03,SeidlSturm03,Strzebonski00} of Collins' algorithm for Cylindrical Algebraic Decomposition (CAD)~\cite{Collins75} make it possible to actually perform nontrivial computations within a reasonable amount of time. They are meanwhile established as valuable tools for solving problems about polynomial inequalities.

In general, the input to CAD is a formula of the form
\[
 \mathrm{Q}_1\,x_1\in \R\ \cdots\ \mathrm{Q}_n\,x_n\in \R: A(x_1,\dots,x_n,y_1,\dots,y_m)
\]
where the $\mathrm{Q}_i$ are quantifiers (either $\forall$ or $\exists$) and $A$
is a boolean combination of polynomial equations and inequalities in the
variables $x_i$ and $y_i$.  The variables $x_i$ are bound by quantifiers, the
variables $y_i$ are free.  Given such a formula, the algorithm computes a
quantifier free formula $B(y_1,\dots,y_m)$ which is equivalent to the input
formula.

A simple example is given by
\[
 \forall\ x\in\R\ \exists\ y\in\R: (x-1)(y-1)>1 \Leftrightarrow x^2+y^2-z^2>0,
\]
where $A(x,y,z)$ is the formula $(x-1)(y-1)>1 \Leftrightarrow x^2+y^2-z^2>0$, the bound
variables are $x,y,$ and $z$ is a free variable. Applied to this formula, the CAD
algorithm may return the quantifier free formula $B(z)=z\leq -1\lor z\geq 1$.
This formula is equivalent to the quantified formula in the sense that for every
real number $z\in\mathbb{R}$ the input formula holds if and only if the output
formula holds.

Applied to a quantified formula with no free variables, CAD will return one of
the two logical constants True or False.
Applied to a formula with only free variables, CAD will produce an equivalent
formula in the same variables which is normalized in a certain sense.

It must be stressed that the equivalence of input and output is not approximate
in any way but completely rigorous. In particular, if CAD applied to a certain
formula $\Phi$ yields the output True, then the trace of the computation
constitutes a lengthy and ugly and insightless but correct and complete and
checkable proof of~$\Phi$. The price to be paid for such a strictly correct
output is that computations may take very long. While it is guaranteed that
every CAD computation will eventually terminate and produce a correct output,
such a guarantee is of little use if the expected runtime exceeds by far our
expected lifetime.

In its original formulation, the dominance relation for Sugeno-Weber t-norms
is an example for a formula which CAD can do in principle but not in practice.
Human interaction is necessary to break the big computation into several smaller
ones, to properly reformulate intermediate results, and to exploit common
properties of different parts of the problem. This has finally lead to a proof
that is explained on five pages, expanded in the Mathematica file available
at http://www.risc.jku.at/people/mkauers/sugeno-weber/proof.nb and
executed in 15 to 30 minutes depending on the actual computation capacities and the
Mathematica version used.

\section{Main results}
\begin{theorem}\label{thm:main1}
Consider the family of Sugeno-Weber t-norms $(\TSW{\lambda})_{\lambda\in\cint{0,\infty}}$. Then, for all $\lambda,\mu\in\cint{0,\infty}$, $\TSW{\lambda}$ dominates $\TSW{\mu}$, $\TSW{\lambda}\gg\TSW{\mu}$, if and only if one of the following conditions holds:
\begin{ilist}
\item $\lambda=0$,
\item $\mu=\infty$,
\item $\lambda=\mu$,
\item $0 < \lambda<\mu \leq 17 + 12 \sqrt2$,
\item $17 + 12 \sqrt2<\mu$ and $0 < \lambda \leq \Bigl(\frac{1 - 3 \sqrt{\mu}}{3-\sqrt\mu}\Bigr)^2$.
\end{ilist}
\end{theorem}

In Fig.~\ref{fig:solutionsets}, we have illustrated for which parameters $\lambda$ for a given parameter $\mu$ it holds that $\TSW{\lambda}$ dominates $\TSW{\mu}$. In Section~\ref{sec:disc} the solution sets will be discussed in more detail.
\begin{figure}
\begin{center}
\includegraphics[width=0.4\textwidth]{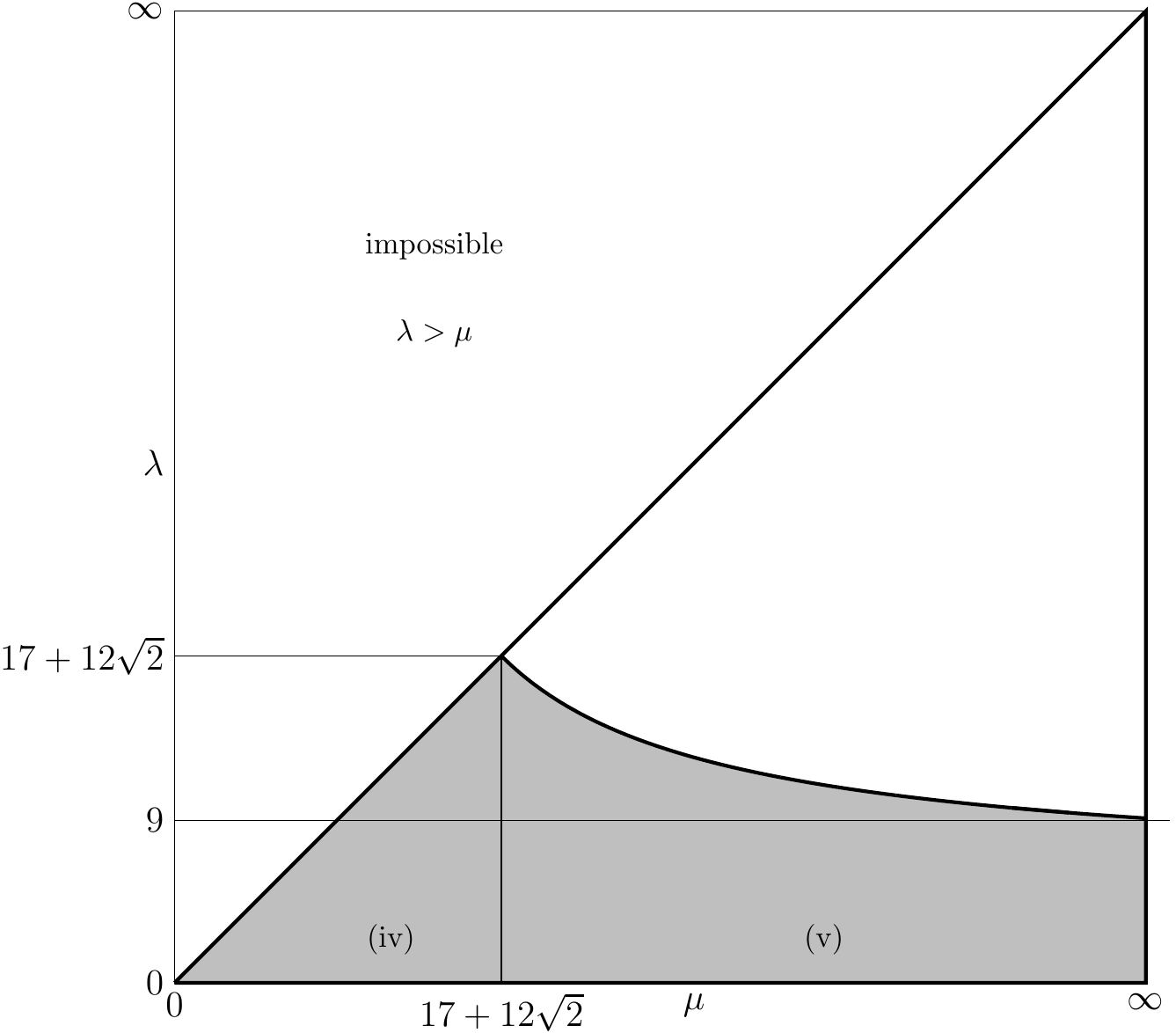}
\end{center}
\caption{Relationship between parameters $\lambda$ and $\mu$ for $\TSW{\lambda}$ dominating $\TSW{\mu}$}
\label{fig:solutionset}
\end{figure}

Based on the characterization of dominance between two members of the family of Sugeno-Weber t-norms, another CAD computation directly applied to the problem of transitivity in the family asserts the transitivity of the relation on the family within less than two seconds. Note also that an alternative proof of the transitivity of dominance in the family of Sugeno-Weber t-norms is given in Section~\ref{sec:AltProofTrans}. In any case, we can state:

\begin{proposition}
Dominance is a transitive, and therefore an order, relation on the set of all Sugeno-Weber t-norms.
\end{proposition}

The remainder of this section is devoted to the proof of Theorem~\ref{thm:main1}, to the necessary steps of reformulating intermediate proof steps, exploiting properties of subparts of the problem, and finding an (equivalent) formulation of the original problem computable and solvable by CAD in reasonable time.

Consider some $\lambda,\mu\in\cint{0,\infty}$. In case that $\lambda=0$, $\mu=\infty$, or $\lambda=\mu$ the result is trivially true since $\TM=\TSW{0}$ dominates all, $\TD=\TSW{\infty}$ is dominated by all t-norms, and every t-norm dominates itself. Since dominance induces order, we may assume w.l.o.g.\ that $0<\lambda<\mu<\infty$. Moreover, the dominance inequality
\[
\TSW{\lambda}(\TSW{\mu}(x,y),\TSW{\mu}(u,v))\geq \TSW{\mu}(\TSW{\lambda}(x,u),\TSW{\lambda}(y,v))
\]
is trivially fulfilled whenever $1\in\{x,y,u,v\}$ or $0\in\{x,y,u,v\}$ such that we can reformulate the remaining proof goal in the following way:

Determine all $\lambda,\mu\in\opint{0,\infty}$ with $\lambda<\mu$ such that for all $x,y,u,v\in\ouint:$
\[
\TSW{\lambda}(\TSW{\mu}(x,y),\TSW{\mu}(u,v))\geq \TSW{\mu}(\TSW{\lambda}(x,u),\TSW{\lambda}(y,v))
\]

or explicitly such that
  \begin{alignat*}{1}
    &\forall\ x,y,u,v\in\ouint:\\
    &\quad
    \max (0,(1-\lambda) \max (0,(1-\mu ) u v+\mu  (u+v-1)) \max (0,(1-\mu ) x y+\mu (x+y-1))\\
    &\qquad+\lambda (\max (0,(1-\mu ) u v+\mu  (u+v-1))+\max (0,(1-\mu ) x y+\mu  (x+y-1))-1))\\
    &\quad\geq \max(0,(1-\mu ) \max (0,(1-\lambda) u x+\lambda (u+x-1)) \max (0,(1-\lambda) v y+\lambda
    (v+y-1))\\
    &\qquad+\mu  (\max (0,(1-\lambda) u x+\lambda (u+x-1))+\max (0,(1-\lambda) v y+\lambda
    (v+y-1))-1)).
  \end{alignat*}

As mentioned earlier this problem is so that a final result can in principle be obtained directly by application of a quantifier elimination algorithm (like, e.g., CAD) for real closed fields. However, in practice this computation would take very long. By a series of appropriate simplifications we can reduce the computation time tremendously. The following simplification steps result in an equivalent quantified formula for which quantifier elimination takes a few minutes only.

  \begin{enumerate}
  \item \emph{Eliminate the outer maxima.}
    The body of the formula in question has the form $\max(0,A)\geq\max(0,B)$.
    It is readily confirmed by hand, or by CAD, that
    \[
      \max(0,A)\geq\max(0,B)
      \iff B\leq0\lor A\geq B
      \iff B\leq0\lor A\geq B>0
    \]
    for all real $A,B$. Applying the last equivalence and making the range restrictions on
    $x,y,u,v$ and $\lambda,\mu$ explicit, we arrive at the equivalent formulation
    \begin{alignat*}1
      &\forall\ x,y,u,v\in\mathbb{R}: 0<\lambda<\mu<\infty\land 0<x<1\land 0<y<1\land 0<u<1\land 0<v<1\\
      &\Rightarrow \Bigl( (1-\mu ) \max (0,(1-\lambda) u x+\lambda (u+x-1)) \max (0,(1-\lambda) v y+\lambda
      (v+y-1))\\
      &\qquad+\mu  (\max (0,(1-\lambda) u x+\lambda (u+x-1))+\max (0,(1-\lambda) v y+\lambda
      (v+y-1))-1)\leq 0\\
      &\quad \lor (1-\lambda) \max (0,(1-\mu ) u v+\mu  (u+v-1)) \max (0,(1-\mu ) x y+\mu (x+y-1))\\
      &\qquad+\lambda (\max (0,(1-\mu ) u v+\mu  (u+v-1))+\max (0,(1-\mu ) x y+\mu  (x+y-1))-1))\\
      &\qquad\geq (1-\mu ) \max (0,(1-\lambda) u x+\lambda (u+x-1)) \max (0,(1-\lambda) v y+\lambda
      (v+y-1))\\
      &\qquad+\mu  (\max (0,(1-\lambda) u x+\lambda (u+x-1))+\max (0,(1-\lambda) v y+\lambda
      (v+y-1))-1)>0\Bigr)
    \end{alignat*}

  \item \emph{Eliminate the inner maxima.}
    The new formula still contains four different maximum expressions:
    \begin{alignat*}{1}
      &\max (0,(1-\mu ) u v+\mu  (u+v-1))\text{ and }
      \max (0,(1-\mu ) x y+\mu  (x+y-1))\text{ in $A$;}\\
      &\max (0,(1-\lambda) u x+\lambda (u+x-1))\text{ and }
      \max (0,(1-\lambda) v y+\lambda (v+y-1))\text{ in $B$.}
    \end{alignat*}
    To get rid of these, observe that if $\Phi(X)$ is a formula depending on a real variable~$X$,
    then the following equivalences are valid
    \begin{alignat*}{1}
      \Phi(\max(0,X))
      &\iff(X\leq 0\lor X>0)\land\Phi(\max(0,X))\\
      &\iff (X\leq 0\land\Phi(0))\lor (X>0\land\Phi(X)).
    \end{alignat*}
    For a formula depending on several real variables, this rewriting yields
    \begin{alignat*}1
            &\kern-3em\Phi(\max(0,X_1),\max(0,X_2),\max(0,X_3),\max(0,X_4))\\
       \iff \bigl(&X_1\leq 0\land X_2\leq 0\land X_3\leq 0 \land X_4\leq 0\land \Phi(0,0,0,0)\\
       {}\lor{} &X_1> 0\land X_2\leq 0\land X_3\leq 0 \land X_4\leq 0\land \Phi(X_1,0,0,0)\\
       {}\lor{} &X_1\leq 0\land X_2> 0\land X_3\leq 0 \land X_4\leq 0\land \Phi(0,X_2,0,0)\\
       {}\lor{} &X_1> 0\land X_2> 0\land X_3\leq 0 \land X_4\leq 0\land \Phi(X_1,X_2,0,0)\\
            &\vdots\\
       {}\lor{} &X_1> 0\land X_2> 0\land X_3> 0 \land X_4> 0\land \Phi(X_1,X_2,X_3,X_4)\bigr).
    \end{alignat*}

    Applying these considerations to our problem, we first put
    \begin{align*}
    X_1&=(1-\lambda)ux+\lambda(u+x-1),\\
    X_2&=(1-\lambda)vy+\lambda(v+y-1),\\
    X_3&=(1-\mu)uv+\mu(u+v-1),\\
    X_4&=(1-\mu)xy+\mu(x+y-1),
    \end{align*}
    so that
    \begin{align*}
    A&=(1-\lambda)\max(0,X_3)\max(0,X_4)+\lambda(\max(0,X_3)+\max(0,X_4)-1),\\
    B&=(1-\mu)\max(0,X_1)\max(0,X_2)+\mu(\max(0,X_1)+\max(0,X_2)-1),
    \end{align*}

    and then arrive at the following equivalent formulation of our problem
    \begin{alignat*}1
      &\forall\ x,y,u,v\in\mathbb{R}:
      % stdhyp
       0<\lambda<\mu<\infty\land 0<x<1\land 0<y<1\land 0<u<1\land 0<v<1\\
      % B <= 0
      &\Rightarrow\Bigl(\bigl( X_1 \leq 0 \land X_2 \leq 0 \land (1-\mu)0\,0+\mu(0+0-1)\leq 0\\
      &\quad\;\lor X_1 > 0 \land X_2 \leq 0 \land (1-\mu)X_1\,0+\mu(X_1+0-1)\leq 0\\
      &\quad\;\lor X_1 \leq 0 \land X_2 > 0 \land (1-\mu)0\,X_2+\mu(0+X_2-1)\leq 0\\
      &\quad\;\lor X_1 > 0 \land X_2 > 0 \land (1-\mu)X_1X_2+\mu(X_1+X_2-1)\leq 0\bigr)\\
      % A >= B > 0
      &\;\;\lor\bigl(X_1\leq0\land X_2\leq 0\land X_3\leq 0\land X_4\leq 0\\
      &\qquad\quad \land (1-\lambda)0\,0+\lambda (0+0-1)\geq(1-\mu)0\,0+\mu(0+0-1)>0\\
      &\quad\;\lor X_1>0\land X_2\leq 0\land X_3\leq 0\land X_4\leq 0\\
      &\qquad\quad \land (1-\lambda)0\,0+\lambda (0+0-1)\geq(1-\mu)X_1\,0+\mu(X_1+0-1)>0\\
      &\quad\;\lor \cdots\\
      &\quad\;\lor X_1>0\land X_2>0\land X_3>0\land X_4>0\\
      &\qquad\quad \land (1-\lambda)X_3X_4 +\lambda(X_3+X_4-1)\geq(1-\mu)X_1X_2+\mu(X_1+X_2-1)>0
      \bigr)\Bigr)
      \rlap{\smash{\raisebox{9\baselineskip}{$\scriptstyle\left.\rule{0pt}{2.25\baselineskip}\right\}B\leq0$}}}
      \smash{\raisebox{3\baselineskip}{$\scriptstyle\left.\rule{0pt}{3.75\baselineskip}\right\}A\geq B>0$}}.
    \end{alignat*}

  \item \emph{Discard redundant clauses.}
    The number of clauses in this last formula can be reduced considerably.
    For example, from
    \[
      0<\lambda<\mu<\infty\land 0<x<1\land 0<y<1\land 0<u<1\land 0<v<1
      \land X_1\leq 0\land X_2\leq 0
      \]
    it follows that $(1-\mu)0\,0+\mu(0+0-1)=-\mu\leq 0$ is trivially true, so that
    the first clause simplifies to $X_1\leq 0\land X_2\leq 0$.

    Similarly, the second and the third clause simplify to $X_1>0\land X_2\leq0$ and $X_1\leq0\land X_2>0$, respectively. (CAD computations
    confirm these assertions quickly.) The last literals of the first three clauses dropped, we can
    simplify the first four clauses, corresponding to $B\leq0$, to
    \begin{alignat*}1
      &(X_1\leq 0\land X_2\leq 0)\lor
       (X_1> 0\land X_2\leq0)\lor
       (X_1\leq 0\land X_2>0)\\
      &\qquad \lor(X_1>0\land X_2>0\land (1-\mu)X_1X_2+\mu(X_1+X_2-1)\leq 0)\\
     &\iff \neg (X_1>0 \land X_2>0)\lor (X_1>0\land X_2>0\land (1-\mu)X_1X_2+\mu(X_1+X_2-1)\leq 0)\\
      &\iff X_1\leq 0 \lor X_2\leq 0 \lor (1-\mu)X_1X_2+\mu(X_1+X_2-1)\leq 0.
    \end{alignat*}
    The simplification of the remaining 16 clauses is complementary: here, all but the last simplify
    to false, and thus these clauses can be dropped altogether. For verifying that a clause
    \begin{alignat*}1
      &0<\lambda<\mu<\infty\land 0<x<1\land 0<y<1\land 0<u<1\land 0<v<1\\
      &\quad\land X_1\diamond 0\land X_2\diamond 0\land X_3\diamond 0\land X_4\diamond 0
       \land A\geq B> 0,
    \end{alignat*}
    is unsatisfiable, with $\diamond$ denoting either $\leq$ or $>$, it is sufficient to show unsatisfiability of the clause with $A\geq B>0$
    replaced by the weaker conditions $A\geq0$ or $B\geq0$. For 15 of the 16 clauses, a CAD computation
    quickly yields false for at least one of these two choices.
    The only surviving clause is the one corresponding to $X_1>0\land X_2>0\land X_3>0\land X_4>0$.
    Dropping all the others and taking into account also the simplified form of the first four
    clauses, we arrive at the equivalent formulation
    \begin{alignat*}1
      &\forall\ x,y,u,v\in\mathbb{R}:0<\lambda<\mu<\infty\land 0<x<1\land 0<y<1\land 0<u<1\land 0<v<1\\
      &\Rightarrow\bigl( X_1 \leq 0\lor X_2\leq 0
      \lor (1-\mu)X_1X_2+\mu(X_1+X_2-1)\leq 0\\
      &\qquad{}\lor X_1>0\land X_2>0\land X_3>0\land X_4>0\\
      &\qquad\qquad{}\land
              (1-\lambda)X_3X_4 +\lambda(X_3+X_4-1)\geq(1-\mu)X_1X_2+\mu(X_1+X_2-1)>0\bigr).
    \end{alignat*}

  \item \emph{Apply some logical simplification.}
    First of all, we may drop the conditions $X_1>0$ and $X_2>0$ from the last clause because
    $X_1\leq0$ and $X_2\leq0$ are part of the disjunction.
    Furthermore, because of
    \begin{alignat*}1
      &0<\lambda<\mu<\infty\land 0<x<1\land 0<y<1\land 0<u<1\land 0<v<1\\
        &\quad{}\land (1-\mu)X_1 X_2+\mu(X_1+X_2-1)>0\Rightarrow X_i>0 \qquad (i=1,2,3,4)
    \end{alignat*}
    as confirmed by a CAD computation within a few minutes,
    also the parts $X_3>0$ and $X_4>0$ in the last clause are redundant and can be dropped.
    Moreover, $(1-\mu)X_1 X_2+\mu(X_1+X_2-1)>0\Rightarrow X_i>0$ ($i=1,2$) is equivalent to
     $X_i\leq 0\Rightarrow (1-\mu)X_1X_2+\mu(X_1+X_2-1)\leq 0$ ($i=1,2$)
    which allows us to discard $X_1\leq0$ and $X_2\leq 0$ from the disjunction.

    Dropping also the $>0$ at the very end of the last clause, which is allowed because
    $(1-\mu)X_1X_2+\mu(X_1+X_2-1)\leq0$ appears in the disjunction, we arrive at the equivalent
    formulation
    \begin{alignat*}1
      &\forall\ x,y,u,v\in\mathbb{R}:0<\lambda<\mu<\infty\land 0<x<1\land 0<y<1\land 0<u<1\land 0<v<1\\
      &\quad\Rightarrow\bigl( (1-\mu)X_1X_2+\mu(X_1+X_2-1)\leq 0\\
      &\qquad{}\lor
              (1-\lambda)X_3X_4 +\lambda(X_3+X_4-1)\geq(1-\mu)X_1X_2+\mu(X_1+X_2-1)\bigr).
    \end{alignat*}
  \item \emph{Apply some algebraic simplification.} In terms of $x,y,u,v$ we have for the second inequality
    \begin{alignat*}1
      &\bigl((1-\lambda)X_3X_4+\lambda(X_3+X_4-1)\bigr)-\bigl((1-\mu)X_1X_2+\mu(X_1+X_2-1)\bigr)\\
      &{}=(\mu-\lambda) \bigl((\mu+\lambda (1-\mu)) (u-1) (v-1) (x-1) (y-1) \\
      &{}\qquad         -((u-1) y-u) ((v-1) x-v)+1\bigr)\geq 0,
    \end{alignat*}
    from which the factor $(\mu-\lambda)$ can be discarded because $0<\lambda<\mu<\infty$ is part of the
    assumptions. Doing in addition the substitutions $x\mapsto 1-x$, $y\mapsto 1-y$, $u\mapsto 1-u$,
    $v\mapsto 1-v$, the last inequality becomes
    \[
      u y + v x (1 - (1-\mu)(1-\lambda) u y)\geq0.
    \]
    The substitutions leave the conditions $0<x<1$, $0<y<1$, $0<u<1$, $0<v<1$ invariant
    but turn the first inequality $(1-\mu)X_1X_2+\mu(X_1+X_2-1)\leq 0$ into
    \begin{alignat*}1
      &u ((\lambda -1) x+1) ((\mu -1) ((\lambda -1) v y+v+y)+1)\\
      &\quad{}+(\mu -1) x ((\lambda -1) v y+v+y)+((\lambda -1) v y+v+y)+x-1 \geq0.
    \end{alignat*}
    This can be simplified further by replacing the subexpression $(\lambda-1)vy+v+y$
    by a new variable~$\tilde{v}$, viz.\ by making the additional substitution $v\mapsto(\tilde{v}-y)/(1+(\lambda-1)y)$.
    This turns the first inequality into
    \[
      u ((\lambda -1) x+1) ((\mu -1) \tilde{v}+1)+(\mu -1) \tilde{v} x+\tilde{v}+x-1\geq0
    \]
    and the second into
    \[
      \frac{\tilde{v} x (1-(\lambda -1) (\mu -1) u y)+y ((\lambda -1) u y ((\mu -1) x+1)+u-x)}{(\lambda
    -1) y+1}\geq0
    \]
    The denominator can be cleared because $(\lambda-1)y+1>0$ is a consequence of the assumptions.
    The substitution also turns the condition $0<v<1$ into $y<\tilde{v}<1+\lambda y$.
    Putting things together, we arrive at the equivalent formulation
    \begin{alignat*}1
      &\forall\ x,y,u,\tilde{v}:0<\lambda<\mu<\infty\land 0<x<1\land 0<y<1\land 0<u<1\land y<\tilde{v}<1+\lambda y\\
      &\quad\Rightarrow\bigl( u ((\lambda -1) x+1) ((\mu -1) \tilde{v}+1)+(\mu -1) \tilde{v} x+\tilde{v}+x-1\geq0\\
      &\qquad{}\lor
         \tilde{v} x (1-(\lambda -1) (\mu -1) u y)+y ((\lambda -1) u y ((\mu -1) x+1)+u-x)\geq0\bigr).
    \end{alignat*}
  \end{enumerate}

  With this last formulation, the quantifier elimination problem can be completed automatically
  within a reasonable amount of time, at least if it is properly input. The order of the quantifiers,
  while logically irrelevant, has a dramatic influence on the runtime. We found that a feasible
  order is $\mu,\lambda,u,y,\tilde{v},x$.
  Mathematica's command Resolve unfortunately reorders the quantifiers internally,
  in this case not to the advantage of the performance. So we have to do the elimination by
  resorting to the low-level CAD command. It is also advantageous to consider the negation of the
  whole formula, and then taking the complement of the result to obtain the desired region for
  $\mu,\lambda$. We thus consider the quantified formula
    \begin{alignat*}1
      &\exists\ x,y,u,\tilde{v}:0<\lambda<\mu<\infty\land 0<x<1\land 0<y<1\land 0<u<1\land y<\tilde{v}<1+\lambda y\\
      &\quad\land u ((\lambda -1) x+1) ((\mu -1) \tilde{v}+1)+(\mu -1) \tilde{v} x+\tilde{v}+x-1<0\\
      &\quad\land
         \tilde{v} x (1-(\lambda -1) (\mu -1) u y)+y ((\lambda -1) u y ((\mu -1) x+1)+u-x)<0.
    \end{alignat*}
  To further improve the performance, we consider the cases $0<\lambda\leq 1$ and $\lambda>1$
  separately. For $0<\lambda\leq1$, the body of the existentially quantified formula is
  unsatisfiable, even when the first big inequality is dropped: A CAD computation quickly asserts
  that
  \begin{alignat*}1
     &0<\lambda<1 \land \lambda<\mu<\infty\land 0<x<1\land 0<y<1\land 0<u<1\land y<\tilde{v}<1+\lambda y\\
     &\quad\land \tilde{v} x (1-(\lambda -1) (\mu -1) u y)+y ((\lambda -1) u y ((\mu -1) x+1)+u-x)<0
  \end{alignat*}
  is equivalent to false. For $\lambda>1$, we proceed in two steps.
  First we compute a CAD only for
  \begin{alignat*}1
    &1<\lambda<\mu<\infty\land 0<x<1\land 0<y<1\land 0<u<1\land y<\tilde{v}<1+\lambda y\\
    &\quad\land u ((\lambda -1) x+1) ((\mu -1) \tilde{v}+1)+(\mu -1) \tilde{v} x+\tilde{v}+x-1<0.
  \end{alignat*}
  This takes about a minute and then gives something which is trivially equivalent to
  \begin{alignat*}1
    &1<\lambda<\mu<\infty \land 0<u<1
    \land 0<y<\tilde{v}<\frac{1-u}{1+(\mu-1)u}
    \land 0<x<\frac{1-u-\tilde{v}-(\mu-1)u\tilde{v}}{(1+(\lambda-1)u)(1+(\mu-1)\tilde{v})}.
  \end{alignat*}
  Denoting this latter formula by $\Phi$, we then compute the CAD of
  \[
    \Phi\land \tilde{v} x (1-(\lambda -1) (\mu -1) u y)+y ((\lambda -1) u y ((\mu -1) x+1)+u-x)<0.
  \]
  This takes about three minutes and then returns
  \begin{alignat*}1
    \mu > 17+12\sqrt2\land \Bigl(\frac{1-3\sqrt\mu}{3-\sqrt\mu}\Bigr)^2<\lambda<\mu<\infty\land(\ldots)
  \end{alignat*}
  where $(\ldots)$ is some messy formula involving $u,y,\tilde{v},x$. The specification of the CAD
  algorithm implies now that the existentially quantified formula above is valid if and only
  if $\mu$ and $\lambda$ satisfy this formula with the $(\ldots)$ part removed.
  Intersecting the complement of this region with the region where $0<\lambda<\mu<\infty$
  (another quick CAD computation), we finally
  obtain
  \[
    \Bigl(0<\mu\leq17+12\sqrt2\land0<\lambda<\mu<\infty\Bigr)\lor\Bigl(\mu>17+12\sqrt2\land0<\lambda\leq\Bigl(\frac{1-3\sqrt\mu}{3-\sqrt\mu}\Bigr)^2\Bigr)
  \]
  as claimed in the beginning.

\section{Discussion of results}\label{sec:disc}
\subsection{Equivalent results}
It is interesting to see that Theorem~\ref{thm:main1} can be expressed in the following equivalent way. This equivalent result shows that the partial results obtained by the sufficient conditions related to the generalized Mulholland inequality as displayed in Proposition~\ref{cor:suffcond} already covered the first four conditions. Condition (v) now closes the missing gap for a full characterization of dominance between two Sugeno-Weber t-norms.

\begin{theorem}\label{thm:main2}
Consider the family of Sugeno-Weber t-norms $(\TSW{\lambda})_{\lambda\in\cint{0,\infty}}$. Then, for all $\lambda,\mu\in\cint{0,\infty}$, $\TSW{\lambda}$ dominates $\TSW{\mu}$ if and only if one of the following conditions holds:
\begin{ilist}
\item $\lambda=0$,
\item $\mu=\infty$,
\item $\lambda=\mu$,
\item $0 < \lambda<\min (\mu,1)$,
\item $0 < \lambda < \mu$ and $1+\sqrt{\lambda\mu}\leq 3(\sqrt\lambda+\sqrt\mu)$.
\end{ilist}
\end{theorem}

\subsection{Solution sets and their properties}
It is further worth to look a bit in more detail at the relationship between the parameters of t-norms $T_\beta$ being dominated by some $T_\alpha$ for some given $\alpha$ of a parametric family of t-norms $(T_\lambda)_{\lambda\in I}$. Figure~\ref{fig:solutionset} visualizes the set of all pairs of parameters $(\lambda,\mu)$ such that $\TSW{\lambda}$ dominates $\TSW{\mu}$. We call such a set \emph{solution set} $\mathcal{S}$, i.e., $\mathcal{S}=\{(\lambda,\mu)\mid T_\lambda\gg T_\mu\}$. In a completely analogous way we have illustrated the solution sets for the parametric families of t-norms as summarized in Table~\ref{tab:families}. The results are displayed in Figure~\ref{fig:solutionsets} and it is immediately obvious that the solution set of the family of Sugeno-Weber t-norms is much more complex than the ones for the other families. Note that for the other families we even do have nice Hasse diagrams whereas for the family of Sugeno-Weber t-norms a nice graphic is, at least so far, still missing.

\begin{figure}[t]
\begin{center}
\begin{tabular}{ccc}
\includegraphics[width=0.3\textwidth]{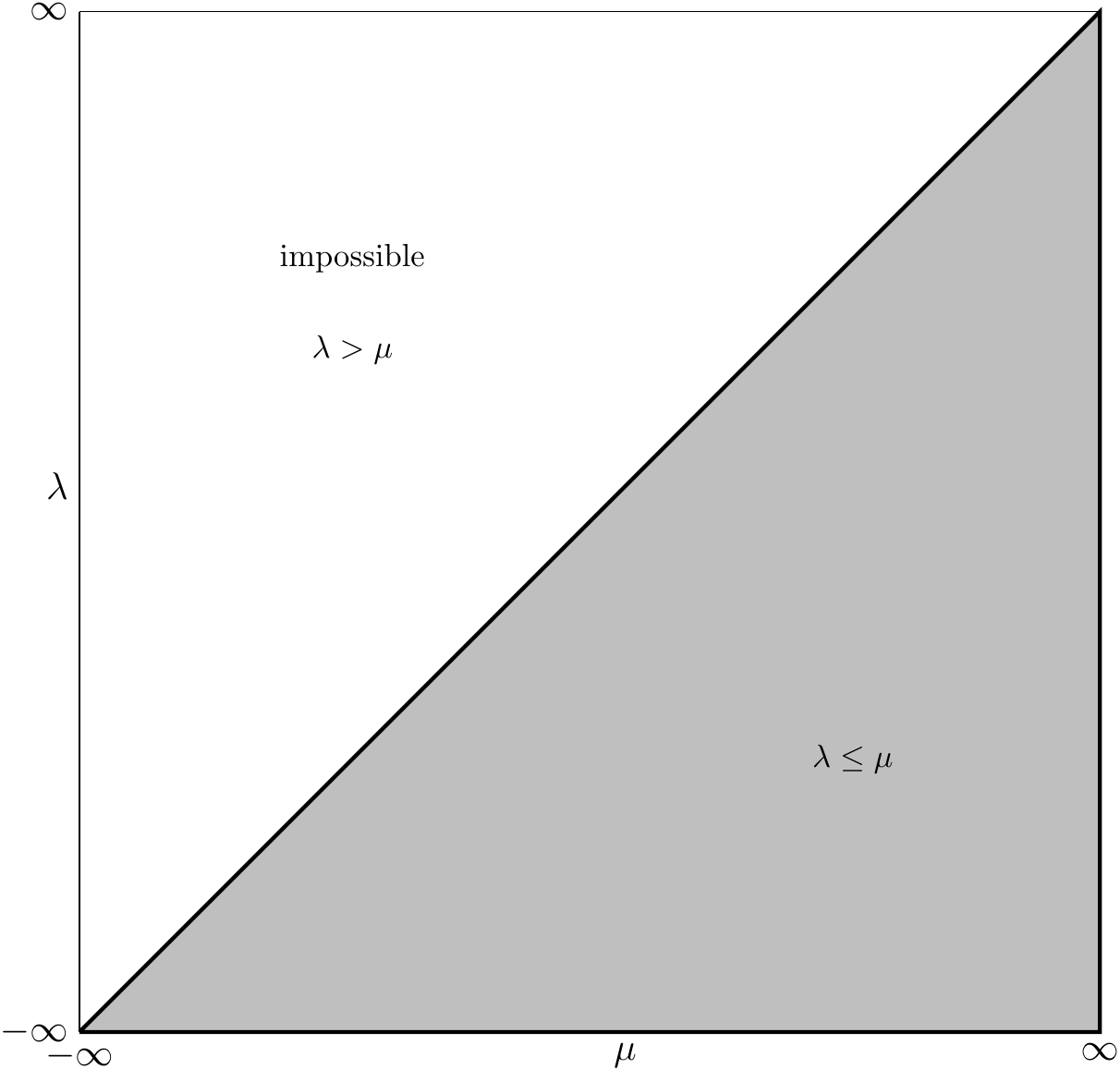}&
\includegraphics[width=0.3\textwidth]{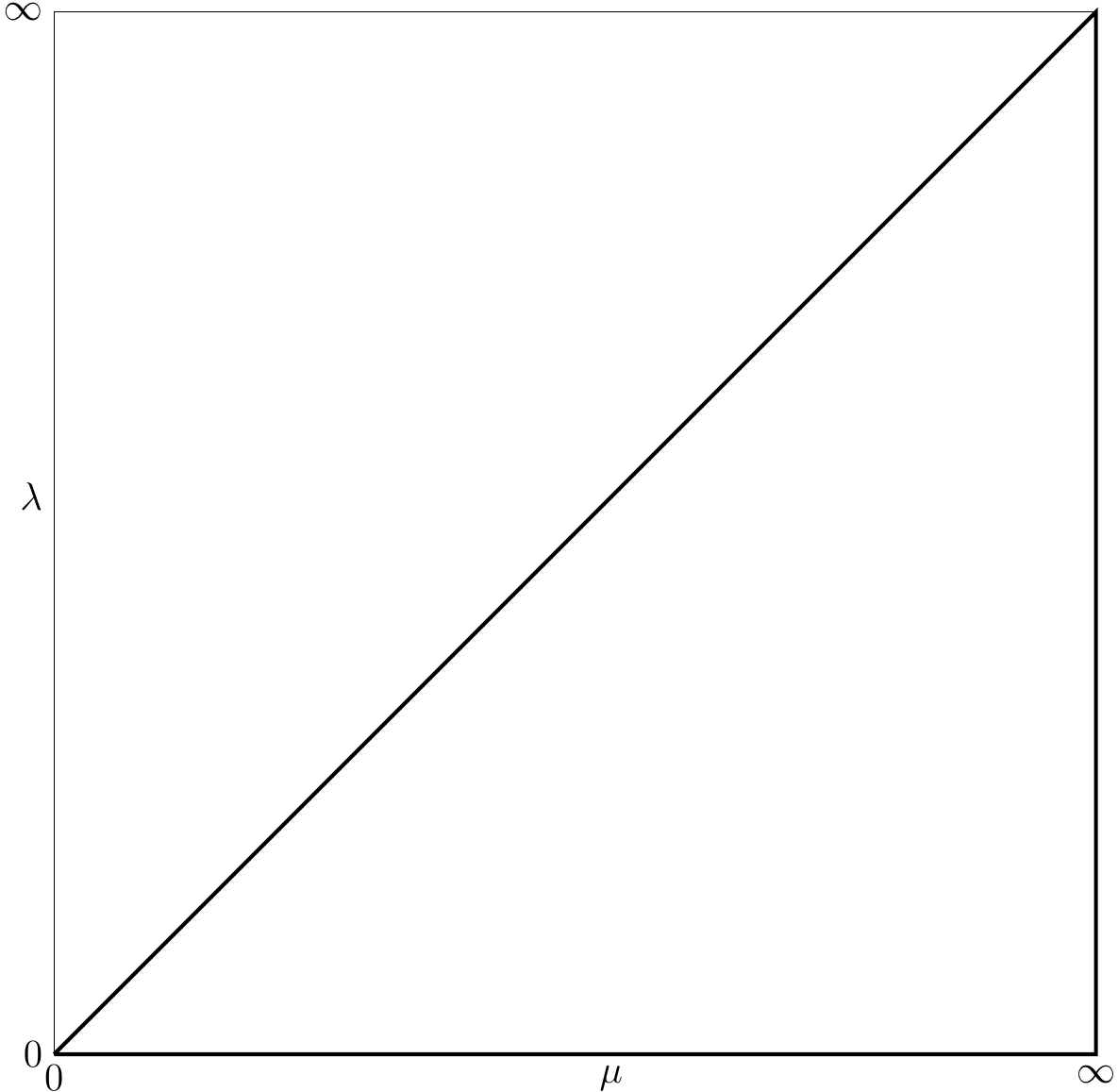}&
\includegraphics[width=0.29\textwidth]{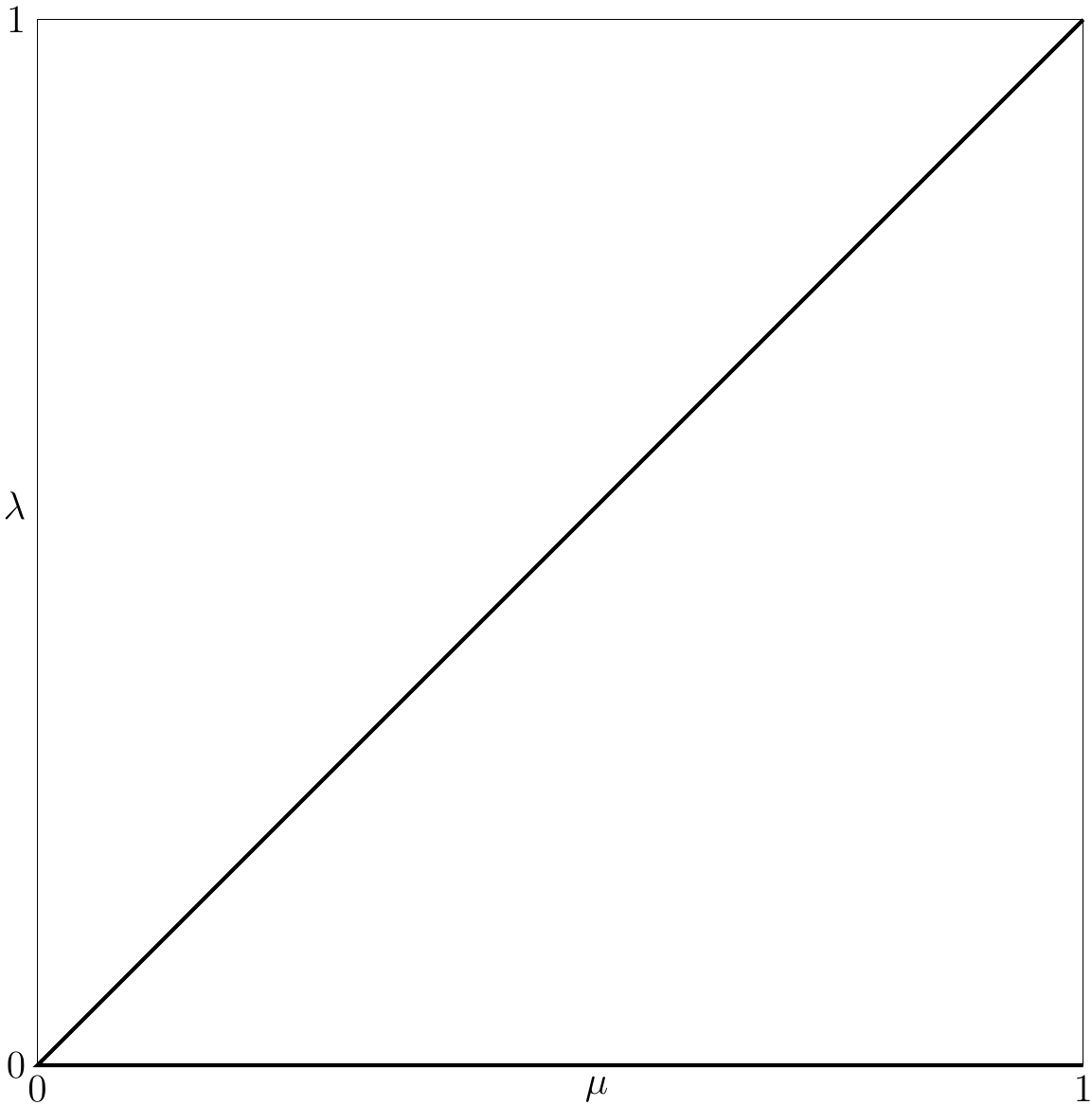}\\
Schweizer-Sklar t-norms & Frank, Hamacher t-norms& Mayor-Torrens, Dubois-Prade t-norms\\[1ex]
\includegraphics[width=0.3\textwidth]{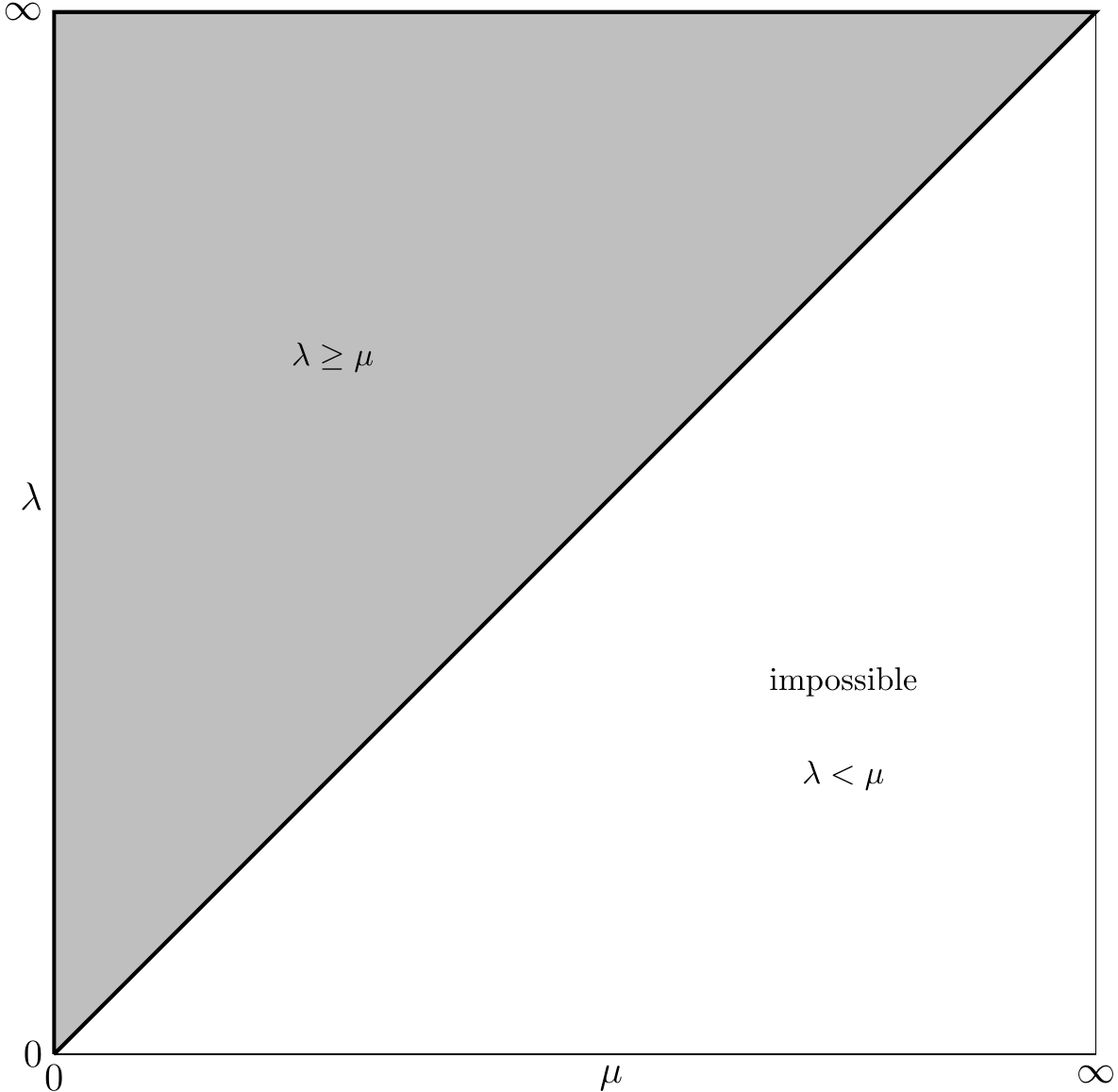}&
\includegraphics[width=0.3\textwidth]{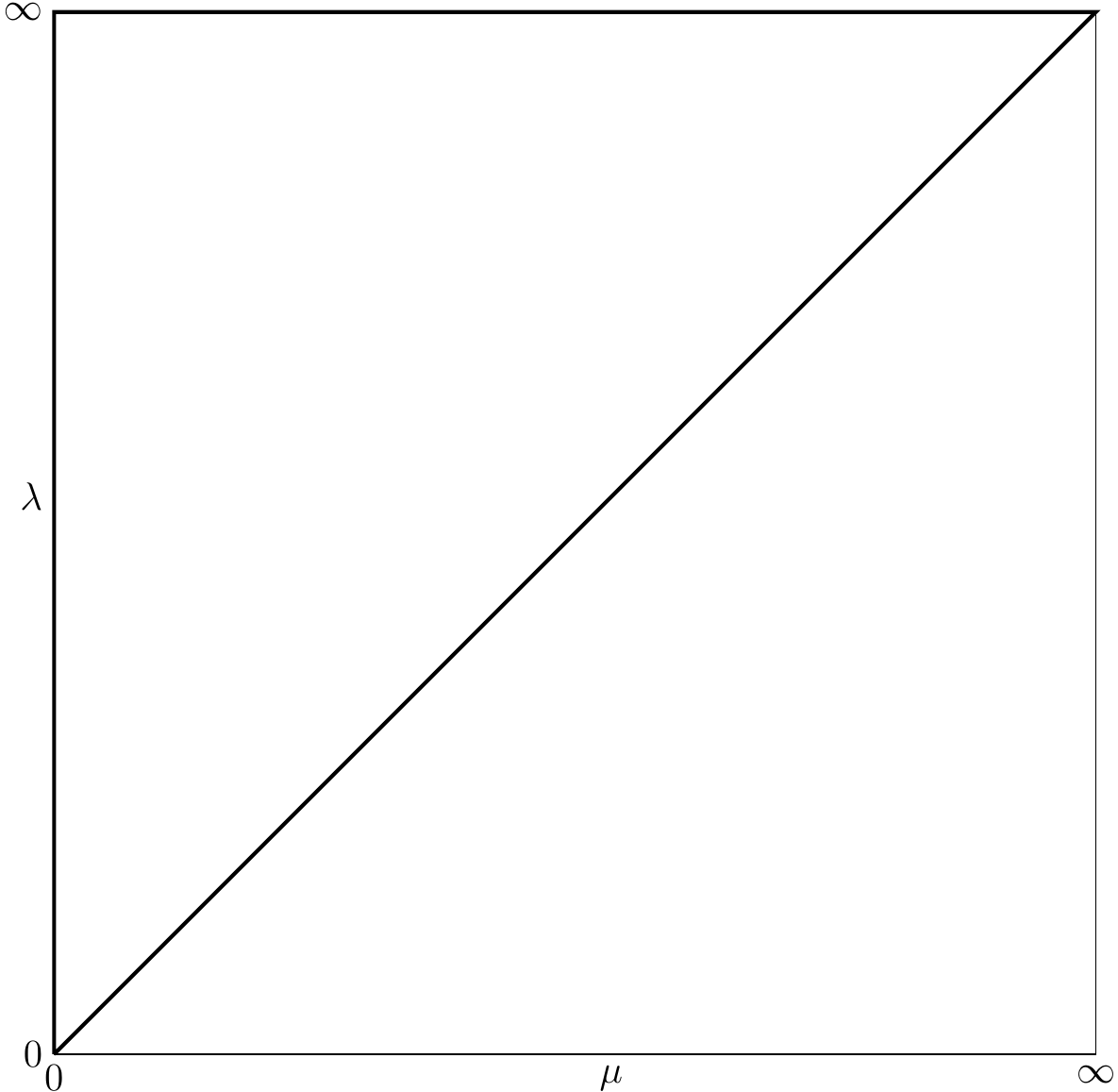}&
\includegraphics[width=0.33\textwidth]{fig_solutionset_result}\\
Acz\'{e}l-Alsina t-norms and others& $T^{\mathbf{9}}$ & Sugeno-Weber t-norms
\end{tabular}
\end{center}
\caption{Schematics for the solution sets for the parametric families of t-norms displayed in Tab.~\ref{tab:families} and for the family of Sugeno-Weber t-norms.    }
\label{fig:solutionsets}
\end{figure}

Therefore we inspect the solution set a bit in more detail. The following function is important for the description of the solution set, so that we briefly discuss its properties:

\begin{corollary}
Consider the function $f\colon\opint{9,\infty}\to \mathbb{R}$ defined, for all $x\in\opint{9,\infty}$, by
\begin{equation}\label{eq:deff}
f(x):=\left(\frac{1-3\sqrt x}{3-\sqrt x}\right)^2.
\end{equation}
Then $f$ is strictly decreasing. Moreover, $f$ is involutive, i.e., for all $x\in\opint{9,\infty}$ we have $f(f(x))=x$.
\end{corollary}

\begin{proof}
The function $f$ is continuous and differentiable. Its first derivative can be computed, for all $x\in\opint{9,\infty}$, as
\[
f'(x)=\frac{-8(1-3\sqrt x)}{\sqrt x (3-\sqrt x)^3}.
\]
For $x>9$ it follows that we have $\sqrt x >3$, $1-3\sqrt x <0$, and $(3-\sqrt x)^3 <0$ and therefore also that $f'(x)<0$ for all $x\in\rint{9,\infty}$, i.e., $f$ is strictly decreasing.

By simple computations it can be easily verified that indeed $f$ is an involution, i.e., if $f(x)=y$ then also $x=f(y)$ and therefore $x=f(y)=f(f(x))$ for arbitrary $x\in\opint{9,\infty}$.
\end{proof}

Note that since $f$ is continuous and strictly decreasing, its range is a convex set. The boundary limits can be computed as $\lim_{x\to 9}f(x)=\infty$ and $\lim_{x\to \infty}f(x)=9$ such that $\Ran_f=\opint{9,\infty}$. Moreover, $\mr$ is one of the fixpoints of $f$.

Let us now turn to a series of results which emphasizes different aspects in the dominance relationship between the members of the family of Sugeno-Weber t-norms.

\begin{corollary}\label{lem0}
For all $\alpha\in\cint{0,9}$, it holds that $\TSW{\alpha}$ dominates $\TSW{\beta}$ for all $\beta\geq \alpha$.
\end{corollary}

\begin{proof}
Let $\alpha$ be an arbitrary real number from $\cint{0,9}$ and choose an arbitrary $\beta\geq \alpha$. If $\alpha=\beta$ or $\beta=\infty$, the dominance relationship trivially holds. We therefore assume that $\alpha<\beta<\infty$. If $\beta\leq \mr$ then $\TSW{\alpha}$ dominates $\TSW{\beta}$ because of Theorem~\ref{thm:main1}~(iv). If $\beta>\mr$, then $f(\beta)> 9$ and therefore $f(\beta)>\alpha$ such that $\TSW{\alpha}$ dominates $\TSW{\beta}$ because of Theorem~\ref{thm:main1}~(v).
\end{proof}

Therefore for all $\alpha\in\cint{0,9}$ the set $\mathcal{D}_\alpha=\{\beta\mid\TSW{\alpha}\gg\TSW{\beta}\}$ is of the form $\cint{\alpha,\infty}$. In case that $\alpha\in\opint{9,\mr}$, $\mathcal{D}_\alpha$ equals $\cint{\alpha,f(\alpha)}\cup\{\infty\}$ as the following Corollary shows.

\begin{corollary}\label{lem1}
For all $\alpha \in \opint{9,\mr}$ there exists some $\beta_\alpha\geq \mr$ such that
\begin{ilist}
\item $\forall \gamma\in\cint{\alpha,\beta_\alpha}\colon \TSW{\alpha} \gg \TSW{\gamma}$,
\item $\forall \delta >\beta_\alpha\colon \TSW{\alpha}\gg \TSW{\delta} \iff \delta=\infty$.
\end{ilist}
\end{corollary}

\begin{proof}
Consider some $\alpha\in\opint{9,\mr}$. Define $\beta_\alpha:=f(\alpha)$.

Since $f$ is continuous and strictly decreasing it obtains its minimal value at its upper boundary. Since $\alpha\leq\mr$ it follows that $\beta_\alpha=f(\alpha)\geq\mr$ for all $\alpha\in\opint{9,\mr}$.
\begin{ilist}
\item
Consider some $\gamma\in\cint{\alpha,\beta_\alpha}$. If  $\alpha\leq\gamma\leq\mr$, $\TSW{\alpha}$ dominates $\TSW{\gamma}$ because of Theorem~\ref{thm:main1}~(iii) and~(iv). For $\mr<\gamma\leq\beta_\alpha$ the decreasingness and involutivness of $f$ imply that $\alpha=f(f(\alpha))=f(\beta_\alpha)\leq f(\gamma)$ such that $\TSW{\alpha}\gg\TSW{\gamma}$ due to Theorem~\ref{thm:main1}~(v).
\item Consider some $\delta> \beta_\alpha$ then if $\delta=\infty$, $\TSW{\alpha}\gg \TSW{\delta}$ trivially holds. Vice versa if $\TSW{\alpha}$ dominates $\TSW{\delta}$, then necessarily $\delta=\infty$, since $f(\delta)<\alpha$.
\end{ilist}
\end{proof}

Note that for all $\alpha\in \opint{9,\mr}$ it holds that $\TSW{\alpha}$ dominates $\TSW{\beta_\alpha}$, since $\beta_\alpha=f(\alpha)>f(\mr)=\mr>\alpha$. For $\alpha=\mr$, $\beta_\alpha=\mr$, such that $\TSW{\alpha}$ dominates $\TSW{\beta_\alpha}$ since $\alpha=\beta_\alpha$.

Finally for all $\alpha\geq \mr$ it holds that $\mathcal{D}_\alpha$ just consists of $\alpha$ and $\infty$.
\begin{corollary}\label{lem3}
For all $\alpha_1,\alpha_2\geq \mr$:
\[
\TSW{\alpha_1}\gg \TSW{\alpha_2} \quad \Rightarrow \quad \alpha_1=\alpha_2 \vee \max(\alpha_1, \alpha_2)=\infty.
\]
\end{corollary}

From the decreasingness of $f$ we immediately conclude:
\begin{corollary}\label{lem2}
For all $\alpha_1,\alpha_2\in \opint{9,\mr}$:
\[
\alpha_1\geq \alpha_2 \quad \Rightarrow \quad \beta_{\alpha_1}\leq \beta_{\alpha_2}.
\]
\end{corollary}

These results allow for an alternative proof of the transitivity of dominance in the family of Sugeno-Weber t-norms.

\subsection{Alternative proof for transitivity}\label{sec:AltProofTrans}
Consider three members $\TSW{a}$, $\TSW{b}$, $\TSW{c}$ of the family of Sugeno-Weber t-norms with arbitrary $a,b,c\in \cint{0,\infty}$. We assume w.l.o.g.\ that $a\neq b\neq c\neq a$. Assume that $\TSW{a}\gg \TSW{b}$ and $\TSW{b}\gg \TSW{c}$ then $a\leq b\leq c$ due to the ordering. We additionally assume that $c <\infty$ for which $\TSW{a}\gg \TSW{c}$ trivially holds. For showing that indeed also $\TSW{a}$ dominates $\TSW{c}$ we distinguish the following cases:
\begin{description}
\item[Case 1.] If $b \geq \mr$, then $c\geq \mr$ and therefore, because of Corollary~\ref{lem3}, $b=c$ or $c=\infty$, the latter being a contradiction.
\item[Case 2.] If $9<b < \mr$, then there exists, because of Corollary~\ref{lem1}, some $\beta_b$ such that $c\in \lint{b, \beta_b}$. Since $a\leq b$, it follows that $\beta_a\geq \beta_b$ (Corollary~\ref{lem2}), and therefore $c\in\lint{b,\beta_b}\subseteq \lint{a,\beta_a}$ such that $\TSW{a}\gg \TSW{c}$ due to Corollary~\ref{lem1}.
\item[Case 3.] If $b\leq 9$, then also $a\leq 9$ such that $\TSW{a}$ dominates $\TSW{b}$.
\end{description}

In all cases $\TSW{a}$ dominates $\TSW{c}$ such that the transitivity of dominance in this family is proven.

%\section{Conclusion}
%We have shown that dominance in the family of Sugeno-Weber t-norms is a transitive relation and we have provided a full characterization of two Sugeno-Weber t-norms being in a dominance relation. Both problems involve polynomial inequalities such that by several steps of formulating intermediate and/or equivalent (sub)goals quantifier elimination and Cylindrical Algebraic Decomposition could most successfully be applied for solving the original problem.
%By the present results, the dominance relation for most of the prominent families of single-parametric t-norms is now laid bare. In all these cases, it constitutes an ordering, although dominance is not a transitive relation on the set of all (continuous) t-norms.

\section*{Acknowledgements}
The authors would like to thank Peter Paule for indicating that CAD might be helpful for solving the problem of dominance in the family of Sugeno-Weber t-norms and as such initiating the collaboration between the authors and Peter Sarkoci for the idea of representing dominating t-norms by solution sets and for fruitful discussions during an early stage of these investigations.

\end{document}